\documentclass[reqno]{amsart}

\usepackage{amssymb,mathtools,float,caption,wrapfig,hyperref,graphicx,url,tikz-cd,enumitem}
\usepackage[scr=boondox]{mathalpha}
\usepackage[abs]{overpic}
\hypersetup{colorlinks,linkcolor={blue},citecolor={blue},urlcolor={blue}}
\graphicspath{ {images/} }
\usepackage[style=alphabetic]{biblatex}
\renewbibmacro{in:}{}
\DeclareFieldFormat*{title}{#1}

\newtheorem{innercustomthm}{Theorem}
\newenvironment{customthm}[1]
  {\renewcommand\theinnercustomthm{#1}\innercustomthm}
  {\endinnercustomthm}

\newtheorem{theorem}{Theorem}[section]

\newtheorem{lemma}[theorem]{Lemma}
\newtheorem{proposition}[theorem]{Proposition}
\newtheorem{corollary}[theorem]{Corollary}

\theoremstyle{definition}
\newtheorem{definition}[theorem]{Definition}
\newtheorem{notation}[theorem]{Notation}

\numberwithin{equation}{section}

\newcommand{\bR}{\mathbb{R}}
\newcommand{\bQ}{\mathbb{Q}}
\newcommand{\bZ}{\mathbb{Z}}
\newcommand{\bF}{\mathbb{F}}
\newcommand{\Gal}{\textup{Gal}}

\newcommand{\cO}{\mathcal{O}}

\newcommand{\cG}{\mathcal{G}}
\newcommand{\cH}{\mathcal{H}}
\newcommand{\cN}{\mathcal{N}}
\newcommand{\cP}{\mathcal{P}}

\newcommand{\cI}{\mathcal{I}}
\newcommand{\fp}{\mathfrak{p}}
\newcommand{\fq}{\mathfrak{q}}
\newcommand{\fa}{\mathfrak{a}}

\newcommand{\fm}{\mathfrak{m}}

\newcommand{\fD}{\mathfrak{D}}
\newcommand{\fP}{\mathfrak{P}}
\newcommand{\fQ}{\mathfrak{Q}}

\newcommand{\rk}{\textup{rk}}
\newcommand{\Cl}{\textup{Cl}}
\newcommand{\cl}{\textup{cl}}

\newcommand{\Ann}{\textup{Ann}}
\newcommand{\nil}{\textup{nil}}

\newcommand{\di}{\hspace{0.02cm}|\hspace{0.02cm}}

\allowdisplaybreaks

\begin{document}

\title[Class groups]{Lower bounds on the $\ell$-rank of ideal class groups}

\thanks{This research is supported by NSF grant 2336000.}

\author{Daniel E. Martin}
\address{Clemson University, O-110 Martin Hall, 2020 Parkway Drive, Clemson, SC}
\email{dem6@clemson.edu}



\date{\today}

\begin{abstract}For a prime number $\ell$ and an extension of number fields $K/F$, we prove new lower bounds on the $\ell$-rank of the ideal class group of $K$ based on prime ramification in $K/F$. Unlike related results from the literature, our bound is supported on prime ideals in $F$ over which at least one (rather than each) prime in $K$ has ramification index divisible by $\ell$. This bound holds with a proviso on the Galois group of the normal closure of $K/F$, which is satisfied by towers of Galois extensions, intermediate fields in nilpotent extensions, and intermediate fields in dihedral extensions of degree $8n$, to name a few. We also use our lower bound to prove a new density result on number fields with infinite class field towers.\end{abstract}

\maketitle

\section{Introduction}\label{sec:1}

Given a prime $\ell\in\bZ$, the $\ell$-rank of an abelian group $A$, denoted $\rk_\ell\, A$, is the dimension of $A/A^\ell$ as a vector space over $\bF_\ell$, where $A^\ell$ is the subgroup of $\ell^\text{th}$ powers. Our goal is prove new lower bounds for $\rk_\ell\,\Cl(K)$, where $\Cl(K)$ is the ideal class group of a number field $K$. 

There is an immense literature on such lower bounds. In place of a survey that would inevitably contain gaps, we reference the reader to some of the most recent papers on the subject; the citations therein cover much of its modern history. For the case when $K$ is a quadratic extension of $\bQ$ see \cite{chatt,gillibert,levin,kishi}, for cubic extensions see \cite{kulkarni}, for dihedral extensions of odd or prime degree see \cite{caputo,lemmermeyer}, and for Galois extensions of prime-power degree see \cite{liu}. For a more general lower bound, with no direct restrictions on $K$, we turn to the work of Roquette and Zassenhaus.

\begin{theorem}[Roquette--Zassenhaus \cite{roquette}]\label{thm:roqzass}For a prime $\ell\in\bZ$ and a number field $K$, let $s_\ell(K)$ be the number of prime ideals in $\bQ$ that become $\ell^\text{th}$ powers in $K$. Then $\emph{rk}_{\ell}\,\emph{Cl}(K)\geq s_\ell(K)-2[K:\bQ]+2$.\end{theorem}

This theorem is the generalization to non-Galois extensions of a result of Brumer and Rosen \cite{rosen}, which Brumer used in combination with the Golod--Shafarevich theorem to give new examples of number fields with infinite class field towers \cite{brumer}. The constant $2[K:\bQ]-2$ given above was improved by Connell and Sussman (as discussed in Section \ref{sec:2}), who also generalized the theorem to relative extensions $K/F$ \cite{connell}.

The point of Theorem~\ref{thm:roqzass}, and indeed parts of all the aforementioned papers, is to improve on the lower bound implied by genus theory, which in many cases is trivial. In order to do so, however, Theorem~\ref{thm:roqzass} requires the presence of many primes $p$ in $\bQ$ that meet a rather severe restriction: $\ell$ must divide \textit{every} ramification index of $p$ in $K$. We can motivate such a restriction by observing the case $\ell=2$ for cubic extensions of $\bQ$. On the one hand, Bhargava and Varma prove that when cubic fields are ordered by discriminant magnitude, the average number of 2-torsion elements in the class group is $5/4$ \cite{bhargava}. On the other hand, the average number of ramified primes grows without bound. So in a cubic extension $K/\bQ$ with ring of integers $\cO_K$, an arbitrary abundance of prime splittings of the form $p\cO_K=\fp_1^2\fp_2$ can, and typically does, contribute almost nothing to $\rk_2\,\Cl(K)$.

We will show that cubic fields are neither the rule nor the exception. As one example, in a quartic extension $K/\bQ$ with dihedral Galois group (of the normal closure) $D_4$, ramified primes, even those of the form $p\cO_K=\fp_1^2\fp_2$ or $\fp_1^2\fp_2\fp_3$, do contribute to $\rk_2\,\Cl(K)$. This is in contrast to genus theory, where primes $p$ that split as $\fp_1^2\fp_2$ or $\fp_1^2\fp_2\fp_3$ do not contribute to the genus number (as the genus field cannot contain $\sqrt{\pm p}$).

More generally, we prove that $\rk_\ell\,\Cl(K)$ grows alongside the number of ramification indices divisible by $\ell$ for a number field with Galois group $D_n$ as long as $4\,|\,n$. (In the aforementioned cubic case in which such growth does not occur, $n=3$.) Dihedral extensions form a proper subset of number fields that possess this property. The full collection of number fields we consider is defined by the Galois group condition below. Before the definition, however, we note a small departure from common notation.

\begin{notation} For an extension of number fields $K/F$, $\Gal(K/F)$ denotes a pair of groups written $G/H$, where $G$ is the Galois group of the closure of $K/F$ and $H$ is the maximal subgroup of $G$ that fixes $K$. We call $K$ a $G/H$\emph{-extension} of $F$.\end{notation}

Note that when $K/F$ is normal, $\Gal(K/F)=G/1$, where $G$ is the usual Galois group. In this case we sometimes write ``$G$" instead of ``$G/1$", as there is never a risk of confusion.

We use $\cl_G(\sigma)$ to denote the conjugacy class of an element $\sigma$ in a group $G$.

\begin{definition}\label{def:divisible}Let $H\leq G$ be groups. We say the pair $G/H$ is $\ell$\emph{-divisible} if there are subgroups $H\leq G_1,...,G_\delta\leq G$ such that given any $\sigma\in G-H$ of order $\ell$, there exists $i\leq \delta$ such that $\cl_{G_i}(\tau)\cap H=\emptyset$ for some $\tau\in\cl_G(\sigma)\cap G_i$. For such a pair, let $\delta_\ell(G/H)$ denote the smallest possible value of $\delta$. We say an extension of number fields $K/F$ is $\ell$\emph{-divisible} if $\Gal(K/F)$ is $\ell$-divisible. For such an extension, let $\delta_\ell(K/F)=\delta_\ell(\Gal(K/F))$.\end{definition}

Returning to the dihedral group $D_n=\langle r,s\mid r^n,s^2,(rs)^2\rangle$, if $4\,|\, n$ then $D_n/\langle s\rangle$ is $2$-divisible with $\delta_2(D_n/\langle s\rangle)=2$, and if $4\nmid n$ then $D_n/\langle s\rangle$ is not $2$-divisible. For an odd prime $\ell$, $\delta_\ell(D_n/\langle s\rangle)$ is 1 if $\ell\,|\, n$ and $0$ otherwise. This is proved in the appendix as a corollary to Theorem~\ref{thm:semidirect}, which gives a general criterion for $\ell$-divisibility of a semidirect product. Also in the appendix, which is essentially an index of ``common" $\ell$-divisible groups and field extensions, are proofs that Galois extensions (Proposition~\ref{prop:Galois}), towers of $\ell$-divisible extensions (Theorem~\ref{thm:fieldtower}), composita of $\ell$-divisible extensions (Theorem~\ref{thm:composite}), and extensions with unique Sylow $\ell$-subgroups (Theorem~\ref{thm:normalS}) are all $\ell$-divisible. 

It is these $\ell$-divisible extensions for which our weaker ramification requirement---$\ell$ divides \textit{some} ramification index of $p$---forces $\ell$-torsion in the class group. The following is proved in Section \ref{sec:2}.

\begin{theorem}\label{thm:intro1}For an extension of number fields $K/F$, let $t_\ell(K/F)$ be the number of prime ideals in $F$ with at least one ramification index divisible by $\ell$, and let $e_\ell(K/F)$ be the exponent of $\ell$ in the prime factorization of $[K:F]$. If $K/F$ is $\ell$-divisible and $t_{\ell}(K/F)\geq 1$, then $$\rk_\ell\,\Cl(K)\geq \frac{t_\ell(K/F)}{\delta_\ell(K/F)}-\rk_{\ell}\,\cO_K^*+\rk_{\ell}\,\cO_F^*-e_\ell(K/F).$$\end{theorem}

Remark that $\rk_\ell\,\cO_K^*$ is either $\rk\,\cO_K^*$ (the free rank of the unit group) or $\rk\,\cO_K^*+1$ depending on whether $K$ contains an $\ell^\text{th}$ root of unity.

We highlight several corollaries of Theorem~\ref{thm:intro1} that are not proved in-text, but are immediate consequences of results from the appendix.

\begin{corollary}\label{cor:intro1}Let $F=F_0\subseteq F_1\subseteq\cdots\subseteq F_n=K$ be number fields such that each $F_i/F_{i-1}$ is Galois. Then $\rk_\ell\,\Cl(K)\geq t_\ell(K/F)/n-\rk_{\ell}\,\cO_K^*+\rk_{\ell}\,\cO_F^*-e_\ell(K/F)$.\end{corollary}

It is interesting to note in the last corollary that $K/F_1$ may well be unramified. Ramification in $F_1/F$ contributes to $\rk_\ell\,\Cl(K)$, but with diminishing returns as $n$ grows.

\begin{corollary}\label{cor:intro2}Let $G/H=\Gal(K/F)$, and suppose $G$ has a unique Sylow $\ell$ subgroup of nilpotency class $n$. Then $\rk_\ell\,\Cl(K)\geq t_\ell(K/F)/n-\rk_{\ell}\,\cO_K^*+\rk_{\ell}\,\cO_F^*-e_\ell(K/F)$.\end{corollary}

\begin{corollary}\label{cor:intro3}Suppose $\Gal(K/F)=(N\rtimes_\varphi H)/H$ with $H$ and $N$ abelian. If for any $h\in H$ of order $\ell$, $\varphi(h)$ is the identity if and only if it is the identity on the Sylow $\ell$-subgroup of $N$, then $\rk_\ell\,\Cl(K)\geq t_\ell(K/F)/2-\rk_{\ell}\,\cO_K^*+\rk_{\ell}\,\cO_F^*-e_\ell(K/F)$.\end{corollary}

For $\ell=2$, the above condition on $\varphi(h)$ is met in the dihedral group of order $2n$ if and only if $4\,|\,n$, hence our prior assertion.

Our second result in Section \ref{sec:2} is somewhat of a corollary to the proof of Theorem~\ref{thm:intro1}. There is potentially significant improvement on Theorem~\ref{thm:intro1}'s lower bound obtained by strengthening the $\ell$-divisibility hypothesis. 

\begin{definition}\label{def:Divisible}We say the pair of groups $G/H$ is \emph{strongly} $\ell$-\emph{divisible} if there are subgroups $H\leq G_1,...,G_\delta\leq G$ such that given any $\sigma\in G-H$ satisfying $\sigma^\ell\in H$ and $|\sigma|=\ell^n$ for some $n$, there exists $i\leq \delta$ such that $\sigma\in G_i$ and $\cl_{G_i}(\sigma)\cap H=\emptyset$. For such a pair, let $\Delta_\ell(G/H)$ denote the smallest possible value of $\frac{1}{\ell}\sum_i[G_i:H]$. We say an extension of number fields $K/F$ is \emph{strongly} $\ell$\emph{-divisible} if $\Gal(K/F)$ is strongly $\ell$-divisible. For such an extension, let $\Delta_\ell(K/F)=\Delta_\ell(\Gal(K/F))$.\end{definition}

In such extensions, our class number bound is no longer supported on ramified primes in the base field $F$ as in Theorem~\ref{thm:intro1} (or like the genus number). We now count primes in $K$.

\begin{theorem}\label{thm:intro2}For an extension of number fields $K/F$, let $T_\ell(K/F)$ be the number of prime ideals in $K$ with ramification index divisible by $\ell$. If $K/F$ is strongly $\ell$-divisible and $T_{\ell}(K/F)\geq 1$, then $$\rk_\ell\,\Cl(K)\geq \frac{T_\ell(K/F)}{\Delta_\ell(K/F)}-\rk_{\ell}\,\cO_K^*+\rk_{\ell}\,\cO_F^*-e_\ell(K/F).$$\end{theorem}

Note that for abelian number fields, Theorem~\ref{thm:intro2} can be deduced (and the ``error term" $\rk_{\ell}\,\cO_K^*+\rk_{\ell}\,\cO_F^*-e_\ell(K/F)$ can be improved; see \cite{furuta}) by computing the relative genus number of the fixed field of the Sylow $\ell$-subgroup of $\Gal(K/F)$.

To get some intuition as to whether Theorem~\ref{thm:intro2} can be expected to improve on Theorem~\ref{thm:intro1}, let us look at the simplest case: normal extensions. Again, this corollary is not proved in-text. It is a consequence of Proposition~\ref{prop:Galois} in the appendix.

\begin{corollary}Suppose $\Gal(K/F)=G/1$, and let $n$ be the number of subgroups of $G$ or order $\ell$. Then $\rk_\ell\,\Cl(K)\geq T_\ell(K/F)/n-\rk_{\ell}\,\cO_K^*+\rk_{\ell}\,\cO_F^*-e_\ell(K/F)$.\end{corollary}

Under the setup above, let us make the heuristic assumption that primes in $F$ with ramification index $\ell$ lie below an average of $\sqrt{|G|/\ell}$ primes in $K$. Then such primes contribute $\sqrt{|G|/\ell}/n$ to the lower bound in Theorem~\ref{thm:intro2}, while only contributing $1$ to the lower bound in Theorem~\ref{thm:intro1}. So as we range over normal extensions of $F$ with Galois group $G$, we might expect improvement with Theorem~\ref{thm:intro2} when $n$ is less than $\sqrt{|G|/\ell}$. (This implicitly uses another heuristic assumption that ramification is, on average, dominated in such number fields by minimal ramification indices. In other words, we expect a ramification index divisible by $\ell$ to equal $\ell$. See Section 4 in \cite{malle2}.)

Section \ref{sec:3} is devoted to consequences of Theorem~\ref{thm:intro1} regarding Hilbert class field towers. Recall that the class field tower of a number field $K$ is the tower of field extensions $K=K_0\subseteq K_1\subseteq K_2\subseteq\cdots$, where $K_i$ is the maximal abelian, unramified extension of $K_{i-1}$. Similarly, to obtain the $\ell$-class field tower we we take maximal abelian, unramified $\ell$-extensions. (See \cite{lemmermeyer2} for a survey.) In 1964, Golod and Shafarevich proved that there are number fields with infinite class field towers. In particular, such number fields have no finite extension with trivial class group, answering the question of Furtw{\"a}ngler in the negative.

\begin{theorem}[Golod--Shafarevich \cite{golod}]\label{thm:golod}Let $K$ be a number field such that $\emph{rk}_\ell\,\Cl(K)>2+2\sqrt{\emph{rk}\,\cO_K^*+1}$ for some prime number $\ell$. Then the $\ell$-class field tower of $K$ is infinite.\end{theorem}

Section \ref{sec:3} begins by applying the Golod--Shafarevich theorem to $\ell$-divisible extensions. The result, stated below, is analogous to a theorem of Brumer \cite{brumer}, who used a lower bound similar to Theorem~\ref{thm:roqzass} (due to him and Rosen \cite{rosen}) to apply the Golod--Shafarevich theorem to number fields with highly composite discriminants. But because prime ramification need not contribute to the perfect $\ell^\text{th}$ powers counted by Theorem~\ref{thm:roqzass}, Brumer could only conclude infinitude of the class field tower when $K/\bQ$ is normal. Theorem~\ref{thm:intro1} allows us to extend Brumer's result to many non-normal extensions. 

In the corollary below, $\omega(n)$ denotes the number of distinct prime factors of a nonzero integer $n$.

\begin{corollary}\label{cor:brumer}Suppose $K/\bQ$ is $\ell$-divisible for all primes $\ell$. Let $D$ be the discriminant of $K$, and let $n$ be the degree of its Galois closure over $\bQ$. If $\omega(\Delta)\geq 4n\omega([K:\bQ])$, then the class field tower of $K$ is infinite.\end{corollary}

A more precise (though harder to compute) bound than $4n\omega([K:\bQ])$ is proved in Corollary~\ref{cor:inftower}, which also generalizes this result to relative extensions. Based on this we might expect that if $G/H$ is $\ell$-divisible for each prime $\ell$, then for any number field $F$, the $G/H$-extensions of $F$ should be dominated by number fields with infinite class field towers because the average number of primes dividing the relative discriminant grows without bound. Such a result appears difficult to prove for arbitrary $G$ as we are left without the aid of class field theory. We can prove this, however, when $G$ is nilpotent (in which case the $\ell$-divisibility hypothesis is automatic by Corollary~\ref{cor:nilpotent}). To state the resulting theorem, we must fix some notation.

\begin{notation}\label{not:malle}Given groups $H\leq G$ and an element $\sigma\in G$, let $n_\sigma$ denote the number of orbits when $\langle\sigma\rangle$ acts on the cosets $G/H$ via left multiplication, and set $a(G/H) = \max_\sigma(1/([G:H]-n_\sigma))$. Also, given $x\in\bR$ and a number field $F$, let $N(F,G/H,x)$ be the number of $G/H$-extensions of $F$ (in some fixed algebraic closure) for which the relative discriminant has absolute norm at most $x$ in magnitude. Among such extensions, let $N_0(F,G/H,x)$ count those with finite class field towers.\end{notation}

\begin{theorem}\label{thm:asymptotic}Let $G/H$ be groups with $G$ nilpotent of nilpotency class $n$, and let $F$ be an extension of $\bQ$ of degree $d$. There is a constant $c(d)$ (depending only on $d$) such that for large $x$, $$N_0(F,G/H,x)<\frac{c(d)x^{a(G/H)}(\log\log x)^{4nd\omega(d)-1}}{\log x}.$$\end{theorem}

The weak Malle conjecture asserts that for any $F$ and $G/H$ there exists a constant $c(F,G)$ such that $N(F,G/H,x)> c(F,G)x^{a(G/H)}$ \cite{malle}. This was proved by Kl{\"u}ners and Malle for nilpotent groups in the special case where $H=1$, meaning the case of Galois extensions \cite{kluners}. In particular, we obtain the following density result.

\begin{corollary}\label{cor:density}With notation as in Theorem \ref{thm:asymptotic}, there is a constant $c(F,G)$ (depending only on $F$ and $G$) such that for large $x$, $$\frac{N_0(F,G/1,x)}{N(F,G/1,x)}<\frac{c(F,G)(\log\log x)^{4nd\omega(d)-1}}{\log x}.$$\end{corollary}

As with Corollary~\ref{cor:brumer}, a more precise (though harder to compute) exponent than $4nd\omega(d)-1$ can be found in Section \ref{sec:3}. Both the previous theorem and its corollary are proved as Theorem~\ref{thm:malle}.

Other density results regarding infinite class field towers can be found for quadratic fields in \cite{gerth,gerth2,mouhib}, for cyclotomic fields in \cite{shparlinski}, and for cyclic extensions of prime degree in \cite{gerthcyclic}.

\section{Proofs of Theorems \ref{thm:intro1} and \ref{thm:intro2}}\label{sec:2}

The following notation is used throughout this section.

\begin{notation}For a prime number $\ell$ and number fields $K/F$, $s_\ell(K/F)$ denotes the number of prime ideals in $F$ whose ramification indices (in $K$) are all divisible by $\ell$; $t_\ell(K/F)$ denotes the number of prime ideals in $F$ with at least one ramification index divisible by $\ell$; $T_\ell(K/F)$ denotes the number of prime ideals in $K$ with ramification index divisible by $\ell$; and $e_\ell(K/F)$ is the exponent of $\ell$ in $[K:F]$.\end{notation}

Our strategy for proving Theorem~\ref{thm:intro1} is to show that if $K/F$ is $\ell$-divisible and $p$ is a prime ideal of $F$ that gets counted by $t_\ell(K/F)$, then there is some intermediate field $F'$ containing a prime $\fp$ over $p$ that gets counted by $s_\ell(K/F')$. We would then like to apply a result like Theorem~\ref{thm:roqzass} to $K/F'$. Unfortunately, the known generalization of Theorem~\ref{thm:roqzass} to relative extensions (Theorem 1 in \cite{connell}) requires $p^n$ to be principal for some $n$ coprime to $\ell$. Even if we assume this of $p$, we cannot conclude the same of some $\fp$ over $p$ in an intermediate field. So we begin by removing this requirement with the next lemma. We note the additional benefit of potentially increasing the number of primes counted in a relative extension when compared to Theorem 1 in \cite{connell}.

\begin{lemma}\label{lem:consus}For a prime number $\ell$ and number fields $K/F$, $$\rk_\ell\,\Cl(K)\geq s_\ell(K/F)-\rk_\ell\,\cO_K^*+\rk_\ell\,\cO_F^*-\rk_\ell\!\left(\frac{F^*\cap K^{*\ell}}{F^{*\ell}}\right).$$\end{lemma}

\begin{proof}Let $G_K$ denote the group of elements $\alpha\in K^*$ such that $\alpha\cO_K=\fa^\ell$ for some fractional ideal $\fa\subset K$, and define $G_F$ analogously. In \cite{connell} (the proof of Theorem 1, but stop about halfway down page 527) it is shown that \begin{align}\label{eq:consus}\rk_\ell\,\Cl(K)&\geq \rk_\ell\,\Cl(F)+\rk_\ell\!\left(\frac{F^*\cap G_K}{G_F}\right)\\\nonumber&\hspace{2\parindent}-\rk_\ell\,\cO_K^*+\rk_\ell\,\cO_F^*-\rk_\ell\!\left(\frac{F^*\cap K^{*\ell}}{F^{*\ell}}\right).\end{align} The proof of our lemma amounts to showing that $s_\ell(K/F)$ does not exceed the sum of the first two terms on right-hand side of this inequality.
    
Let $s=s_\ell(K/F)$, let $p_1,...,p_s\subset \cO_F$ denote the prime ideals that become $\ell^\text{th}$ powers in $K$, let $\cI_p$ denote the subgroup of fractional ideals generated by $p_1,...,p_s$, and let $\cP_F$ denote the group of principal fractional ideals in $F$. The canonical exact sequence $$1\to \cI_p\cap \cP_F\to \cI_p\to \frac{\cI_p}{\cI_p\cap \cP_F}\to 1$$ induces an exact sequence of vector spaces over $\bF_\ell$: $$1\to\frac{\cI_p\cap \cP_F}{\cI^\ell_p\cap \cP_F}\to \frac{\cI_p}{\cI_p^\ell}\to\frac{\cI_p}{\cI_p^\ell(\cI_p\cap \cP_F)}\to 1.$$ This implies \begin{equation}\label{eq:sequals}\rk_\ell\!\left(\frac{\cI_p}{\cI_p^\ell(\cI_p\cap \cP_F)}\right)+\rk_\ell\!\left(\frac{\cI_p\cap \cP_F}{\cI^\ell_p\cap \cP_F}\right)=\rk_\ell\!\left(\frac{\cI_p}{\cI_p^\ell}\right),\end{equation} which is precisely $s$ by definition of $\cI_p$. Let us consider the two dimensions on the left-hand side above one at a time in the order they appear.

First let $\varphi:\cI_p\to\Cl(K)$ denote the map sending an ideal to its ideal class. We compose this with the quotient map restricted to the image of $\varphi$: $$\cI_p\xrightarrow{\varphi}\varphi(\cI_p)\to\frac{\varphi(\cI_p)}{\varphi(\cI_p)^\ell}.$$ An ideal $a\in \cI_p$ belongs to the kernel of the composition above if and only if $a$ and $b^\ell$ belong to the same ideal class for some $b\in \cI_p$. Writing $a=b^\ell(a/b^\ell)$, we see that the kernel is precisely $\cI_p^\ell(\cI_p\cap \cP_F)$. So we get an injection $$\frac{\cI_p}{\cI_p^\ell(\cI_p\cap \cP_F)}\hookrightarrow \frac{\varphi(\cI_p)}{\varphi(\cI_p)^\ell},$$ which implies \begin{equation}\label{eq:rkcl}\rk_\ell\!\left(\frac{\cI_p}{\cI_p^\ell(\cI_p\cap \cP_F)}\right)\leq\rk_\ell\,\varphi(\cI_p)\leq\rk_\ell\,\Cl(F).\end{equation} 
    
Second, if $\alpha\in F$ is such that $\alpha\cO_F\in \cI_p$ then $\alpha\cO_K=\fa^\ell$ for some nonzero fractional ideal $\fa\subset K$; indeed, generators for $\cI_p$ all become $\ell^\text{th}$ powers in $K$. This is the defining condition for the group $G_K$, so $\alpha\in G_K$. Since a generator for $\alpha\cO_F$ is defined up to multiplication by a unit, we have an injection \begin{equation}\label{eq:inj}\cI_p\cap \cP_F\hookrightarrow \frac{F^*\cap G_K}{\cO_F^*}\end{equation} given by $\alpha\cO_F\mapsto \alpha\cO_F^*$. And because $\cO_F^*\subset G_F$, we have a quotient map \begin{equation}\label{eq:quo}\frac{F^*\cap G_K}{\cO_F^*}\to \frac{F^*\cap G_K}{G_F}.\end{equation} By definition of $G_F$ (and unique factorization of ideals), the kernel of the composition of (\ref{eq:inj}) and (\ref{eq:quo}) is $\cI_p^\ell\cap \cP_F$. Thus $(\cI_p\cap \cP_F)/(\cI_p^\ell\cap \cP_F)$ injects into $(F^*\cap G_K)/G_F$, implying \begin{equation}\label{eq:rkf}\rk_\ell\!\left(\frac{\cI_p\cap \cP_F}{\cI_p^\ell\cap \cP_F}\right)\leq \rk_\ell\!\left(\frac{F^*\cap G_K}{G_F}\right)\!.\end{equation}

Combining (\ref{eq:sequals}), (\ref{eq:rkcl}), and (\ref{eq:rkf}) shows that $$\rk_\ell\,\Cl(F)+\rk_\ell\!\left(\frac{F^*\cap G_K}{G_F}\right)\geq \rk_\ell\!\left(\frac{\cI_p}{\cI_p^\ell}\right) = s_\ell(K/F).$$ This, in turn, combines with (\ref{eq:consus}) to complete the proof.\end{proof}

Connell and Sussman demonstrate by example how a well-chosen base field $F$ can significantly increase the value of their lower bound due to the appearance of $\rk_\ell\,\Cl(F)$ in (\ref{eq:consus}). We have removed this term in Lemma~\ref{lem:consus}, but it is still important to allow the base field to vary in our result below. This is because $K$ may be $\ell$-divisible (recall Definition \ref{def:divisible}) over one choice of base field but not another. For example, let $L/\bQ$ be a normal extension with Galois group equal to the symmetric group $S_5$. Let $H$ be a copy of the Klein 4-group in $S_5$, and let $K$ be the fixed field of $H$. Then $K/\bQ$ is not 2-divisible (as a quick check shows that $S_5/H$ is not 2-divisible). However, if $F$ is the fixed field of the copy of $S_4$ (also $D_4$ works) in $S_5$ that contains $H$, then $K/F$ is normal and thus 2-divisible. Note that $[L:K]$ and $[K:F]$ are even in this example, so there is potential for $t_\ell(K/F)$ to be positive, and Theorem \ref{thm:main1} would only apply to $K/F$, not $K/\bQ$.

\begin{customthm}{\ref{thm:intro1}}\label{thm:main1}For an $\ell$-divisible extension $K/F$, if $t_{\ell}(K/F)\geq 1$ then $$\rk_\ell\,\Cl(K)\geq \frac{t_\ell(K/F)}{\delta_\ell(K/F)}-\rk_{\ell}\,\cO_K^*+\rk_{\ell}\,\cO_F^*-e_\ell(K/F).$$\end{customthm}

\begin{proof}Let $L$ denote the normal closure of $K/F$ (in some fixed algebraic closure), let $\delta = \delta_\ell(K/F)$, let $H=\Gal(L/K)$ and $G=\Gal(L/F)$, let $H\leq G_1,...,G_\delta\leq G$ denote the subgroups defined in Definition \ref{def:divisible}, and let $F\subseteq F_1,...,F_\delta\subseteq K$ denote their fixed fields.

Consider a pair of prime ideals $\fp\subset\cO_K$ and $p=\fp\cap F$ for which $\ell$ divides the ramification index $e(\fp\di p)$. Fix any prime $\fP\subset\cO_L$ above $\fp$. Using $I$ to denote an inertia group, \begin{equation}\label{eq:inertia}\frac{|I(\fP\di p)|}{|H\cap I(\fP\di p)|}=e(\fp\di p)\end{equation} shows that a Sylow $\ell$-subgroup of $I(\fP\di p)$ properly contains its intersection with $H$. In particular, there exists $\sigma\in I(\fP\di p)-H$ such that $|\sigma|=\ell^{n+1}$ for some $n\geq 0$. The elements of $\cl_G(\sigma^{\ell^n})$ generate a normal subgroup of $G$. Since $L$ is the smallest normal extension of $F$ containing $K$, there are no nontrivial normal subgroups of $G$ contained in $H$. Thus $\cl_G(\sigma^{\ell^n})\not\subseteq H$. Applying the $\ell$-divisibility hypothesis, we conclude that for some index $i=i_p\leq \delta$, there exists $\tau\in \cl_G(\sigma^{\ell^n})\cap G_i$ satisfying \begin{equation}\label{eq:clempty}\cl_{G_i}(\tau)\cap H=\emptyset.\end{equation}
    
Let $\pi\in G$ satisfy $\pi\sigma^{\ell^n}\pi^{-1}=\tau$. (If there are multiple choices for $\sigma$, $i$, $\tau$, or $\pi$, pick any one.) Let $q=q_p=\pi(\fP)\cap F_i$, let $\fq$ be some prime above $q$ in $K$, and let $\fQ$ be a prime above $\fq$ in $L$. Assume by way of contradiction that $\ell\nmid e(\fq\di q)$. Then replacing $\fP\di \fp\di p$ in (\ref{eq:inertia}) with $\fQ\di\fq\di q$ shows that $\ell$ cannot divide $|I(\fQ\di q)|/|H\cap I(\fQ\di q)|$. In particular, a Sylow $\ell$-subgroup of $H\cap I(\fQ\di q)$, call one $S$, is also a Sylow $\ell$-subgroup of $I(\fQ\di q)$. Now, \begin{align*}\tau &\in G_i\cap \pi I(\fP\di p)\pi^{-1}\hspace{\parindent}\text{since }\tau = \pi\sigma^{\ell^n}\pi^{-1}\text{ and }\sigma\in I(\fP\di p)\\&= G_i\cap I(\pi(\fP)\di p)\\&=I(\pi(\fP)\di \pi(\fP)\cap F_i)\\& = I(\pi(\fP)\di q)\hspace{\parindent}\text{by definition of }q,\end{align*} and $I(\pi(\fP)\di q)$ is conjugate to $I(\fQ\di q)$ in $G_i$. Furthermore, whatever conjugate of $\tau$ lies in $I(\fQ\di q)$ has order $|\tau|=\ell$, so it belongs to some Sylow $\ell$-subgroup of $I(\fQ\di q)$. Since any Sylow $\ell$-subgroup of $I(\fQ\di q)$ is conjugate to $S$ (in $I(\fQ\di q)$ and therefore in $G_i$), we see that $\cl_{G_i}(\tau)\cap S$ is nonempty. But then $S\leq H$ contradicts (\ref{eq:clempty}). Thus $\ell\,|\,e(\fq\di q)$. Because $\fq$ was an arbitrary prime above $q$, this implies $q\cO_K$ is an $\ell^\text{th}$ power. 
    
Among the $t_\ell(K/F)$ primes $p$ in $F$ under consideration, the pigeon-hole principal says there must be at least $\lceil t_\ell(K/F)/\delta\rceil$ of them for which the indices $i_p$ produced above are all equal, say to $i_0\leq \delta$. We apply Lemma~\ref{lem:consus} to $K/F_{i_0}$ to obtain \begin{align*}\rk_\ell\,\Cl(K)&\geq s_\ell(K/F_{i_0})-\rk_\ell\,\cO_K^*+\rk_\ell\,\cO_{F_{i_0}}^*-\rk_\ell\!\left(\frac{F_{i_0}^*\cap K^{*\ell}}{F_{i_0}^{*\ell}}\right)\\ &\geq\frac{t_\ell(K/F)}{\delta_\ell(K/F)}-\rk_\ell\,\cO_K^*+\rk_\ell\,\cO_F^*-\rk_\ell\!\left(\frac{F_{i_0}^*\cap K^{*\ell}}{F_{i_0}^{*\ell}}\right).\end{align*} The final term above is the power of $\ell$ that divides the degree over $F_{i_0}$ of the smallest field containing all $\alpha\in K$ such that $\alpha^\ell\in F_{i_0}$. This exponent is bounded from above by $e_{\ell}(K/F_{i_0})\leq e_{\ell}(K/F)$, which completes the proof.\end{proof}

Regarding the end of the proof, note that it is possible for a tower of fields $K/F_{i_0}/F$ to satisfy $$\rk_\ell\!\left(\frac{F_{i_0}^*\cap K^{*\ell}}{F_{i_0}^{*\ell}}\right)> \rk_\ell\!\left(\frac{F^*\cap K^{*\ell}}{F^{*\ell}}\right).$$ As a simple example, take any odd prime $\ell$, let $F$ denote the cyclotomic field $\bQ(\zeta_\ell)$, and let $K'$ be any Galois extension of $\bQ$ with dihedral Galois group $D_\ell$ such that $K'\cap F=\bQ$. Then the inequality above holds with $K=K'F$ and $F_{i_0}$ as the quadratic extension of $F$ in $K$. Indeed, since $[K:F_{i_0}]=\ell$, Kummer theory tells us $\rk_\ell\,(F_{i_0}^*\cap K^{*\ell}/F_{i_0}^{*\ell}) = 1$. But $\Gal(K/F)\simeq\Gal(K'/K'\cap F)\simeq D_\ell$, so there is no normal extension of $F$ of degree $\ell$ in $K$. Thus $\rk_\ell\,(F^*\cap K^{*\ell}/F^{*\ell}) = 0$.

Also remark that Theorem \ref{thm:main1} \textit{appears} to provide nontrivial bounds on $\rk_\ell\,\Cl(K)$ even when $\ell\nmid[K:F]$, but this is not the case. To see why, we cite Lemma~\ref{lem:syls} from the appendix: if $K/F$ is $\ell$-divisible and $\ell\nmid[K:F]$, then $H$ contains all Sylow $\ell$-subgroups of $G$, where $G/H=\Gal(K/F)$. With this in mind, suppose $t_\ell(K/\bQ)\geq 1$ so that Theorem \ref{thm:main1} might apply. Then (\ref{eq:inertia}) shows that $H$ does not contain all Sylow $\ell$-subgroups of $G$. So by the contrapositive of Lemma~\ref{lem:syls}, $K/F$ cannot be $\ell$-divisible if $\ell\nmid[K:\bQ]$. Let us highlight this observation to be used Section \ref{sec:3}.

\begin{proposition}\label{prop:mustdivide}If $K/F$ is $\ell$-divisible and $t_\ell(K/F)\geq 1$, then $\ell\,|\,[K:F]$.\qed\end{proposition}

We return to the proof of Theorem~\ref{thm:main1} in order to observe a potential improvement. There is a lack of control over which $G$-conjugate of $\fp$, denoted $\fq$ in the proof, lies above a prime that becomes an $\ell^\text{th}$ power in $K$. This lack of control comes from Definition \ref{def:divisible}'s assertion that ``$\cl_{G_i}(\tau)\cap H=\emptyset$ for \textit{some} $\tau\in\cl_G(\sigma)\cap G_i$." Definition \ref{def:Divisible} strengthens this part of the hypothesis, leading to the potentially greater bound below.

\begin{customthm}{\ref{thm:intro2}}\label{thm:main2}For a strongly $\ell$-divisible extension $K/F$, if $T_{\ell}(K/F)\geq 1$ then $$\rk_\ell\,\Cl(K)\geq \frac{T_\ell(K/F)}{\Delta_\ell(K/F)}-\rk_{\ell}\,\cO_K^*+\rk_{\ell}\,\cO_F^*-e_\ell(K/F).$$\end{customthm}

\begin{proof}Let $L$, $G$, and $H$ be as in the proof of Theorem~\ref{thm:main1}, let $H\leq G_1,...,G_\delta\leq G$ denote the subgroups defined in Definition \ref{def:Divisible}, and let $F\subseteq F_1,...,F_\delta\subseteq K$ denote their fixed fields.
    
Consider a pair of prime ideals $\fp\subset\cO_K$ and $p=\fp\cap F$ for which $\ell$ divides the ramification index $e(\fp\di p)$. Let $\fP$ be a prime above $\fp$ in $L$. By (\ref{eq:inertia}), There must be some $\sigma\in I(\fP\di p)-H$ with order $\ell^n$ for some $n\geq 1$. By taking a power of $\sigma$ if necessary, we may assume that $\sigma^\ell\in H$. Strong $\ell$-divisibility now allows us to apply the remainder of the proof of Theorem~\ref{thm:main1} to $\sigma$ itself rather than some $\tau$ that is conjugate to a power of $\sigma$. More precisely, if $i=i_\fp\leq \delta$ is the index (assumed to exist) satisfying $\sigma\in G_i$ and $\cl_{G_i}(\sigma)\cap H=\emptyset$, then the prime ideal $q=q_\fp=\fp\cap F_i$ becomes an $\ell^{\text{th}}$ power in $K$ by the exact same argument as before. Observe that this bounds the number of distinct primes above $q$ in $K$ by $[K:F_i]/\ell$. So the function \begin{equation}\label{eq:nto1}\fp\mapsto q_\fp\text{ is at most }[K:F_{i_\fp}]/\ell\,\text{-}\,\text{to}\,\text{-}\,1.\end{equation}

For each $i=1,...,\delta$, let $s_i$ denote the number of distinct primes $q_{\fp}$ as $\fp$ ranges over all primes counted by $T_\ell(K/F)$ that satisfy $i_\fp = i$. We have \begin{align*}T_\ell(K/F)&=\sum_{i=1}^\delta\sum_{q\subset F_i}\big|\{\fp:q_{\fp}=q\}\big|\hspace{\parindent}\text{by grouping }\fp\text{ according to }i_\fp\text{ then }q_\fp\\ &\leq \sum_{i=1}^\delta\frac{s_i[K:F_i]}{\ell}\hspace{\parindent}\text{by (\ref{eq:nto1}) and the definition of }s_i\\&\leq \max_{1\leq i\leq \delta}\!(s_i)\sum_{i=1}^\delta\frac{[G_i:H]}{\ell}=\max_{1\leq i\leq \delta}\!(s_i)\Delta_\ell(K/F).\end{align*} By comparing either end of the chain of inequalities above, we see that $s_{i_0}\geq T_{\ell}(K/F)/\Delta_\ell(K/F)$ for at least one $i_0\leq \delta$. We apply Lemma~\ref{lem:consus} to the extension $K/F_{i_0}$ just as we did at the end of the proof of Theorem~\ref{thm:main1}.\end{proof}

Three corollaries to Theorems \ref{thm:main1} and \ref{thm:main2} can be found in the introduction, and others can be deduced from results in the appendix. However, the list of $\ell$-divisible families of groups and field extensions in the appendix is far from complete, and for those families that do appear, the accompanying upper bounds on $\delta_\ell$ and $\Delta_\ell$ are often only sharp in special cases. So for specific applications of Theorem~\ref{thm:main1} or \ref{thm:main2}, best is to simply prove $\ell$-divisibility or strong $\ell$-divisibility and compute exact values of $\delta_\ell$ or $\Delta_\ell$ by hand or with the aid of a computer algebra system like GAP \cite{GAP}.

\section{Class field towers}\label{sec:3}

For Galois extensions of $\bQ$ of fixed degree $d$, Brumer was the first to observe that a highly composite discriminant $D$ is sufficient to guarantee an infinite class field tower \cite{brumer}. Indeed, the pigeon-hole principle forces at least $\omega(D)/\omega(d)$ primes in $\bQ$ to have ramification indices that share a nontrivial common factor, some prime $\ell\,|\,d$. Theorem~\ref{thm:roqzass} then bounds $\rk_\ell\,\Cl(K)$ from below, and the Golod--Shafarevich theorem (Theorem~\ref{thm:golod}) completes the proof. 

A similar corollary falls out of Theorem~\ref{thm:main1}. For a number field $F$, let $\omega_F$ denote the number of distinct prime ideals of its argument in $\cO_F$.

\begin{corollary}\label{cor:inftower}Let $K/F$ be an extension of number fields of degree $d>1$ and relative discriminant $\fD$. If $K/F$ is $\ell$-divisible for every prime number $\ell$ and $$\left\lceil\frac{\omega_F(\fD)}{\omega_\bQ(d)}\right\rceil\geq \max_{\ell\,|\,d}\Big(\delta_\ell(K/F)\big(2+2\sqrt{\rk\,\cO_K^*+1}+\rk_\ell\,\cO_K^*-\rk_\ell\,\cO_F^*+e_\ell([K:F])\big)\Big),$$ then the $\ell$-class field tower of $K$ is infinite for some prime $\ell\,|\,d$.\end{corollary}

\begin{proof}The combination of Theorems~\ref{thm:main1} and \ref{thm:golod} tells us that $K$ has an infinite $\ell$-class field tower if \begin{equation}\label{eq:golod}\frac{t_\ell(K/F)}{\delta_\ell(K/F)}\geq 2+2\sqrt{\rk\,\cO_K^* + 1} + \rk_\ell\,\cO_K^* - \rk_\ell\,\cO_F^*+e_\ell(d).\end{equation} So to complete the proof, we need only show that $t_\ell(K/F)\geq \lceil\omega_F(\fD)/\omega_\bQ(d)\rceil$ for some prime $\ell$.

Each of the $\omega_F(\fD)$ prime ideals $p$ dividing $\fD$ lies below at least one $\fp\subset K$ for which $e(\fp\di p)>1$. By Lemma~\ref{prop:mustdivide}, any prime divisor of $e(\fp\di p)$ is one of the $\omega_\bQ(d)$ prime divisors of $d$. Thus by the pigeon-hole principal, there exists a prime $\ell$ such that at least $\lceil\omega_F(\fD)/\omega_\bQ(d)\rceil$ primes $p$ in $F$ satisfy $\ell\,|\,e(\fp\di p)$ for at least one $\fp$ above $p$ in $K$. In other words, $t_\ell(K/\bQ)\geq \lceil\omega_F(\fD)/\omega_\bQ(d)\rceil$ as desired.\end{proof}

Let us use this to verify (a relative version of) the simpler bound stated as Corollary~\ref{cor:brumer} in the introduction.

\begin{lemma}If $G/H$ is $\ell$-divisible, then $\delta_\ell(G/H)\leq 1+(|H|-1)/(\ell-1)$.\end{lemma}

\begin{proof}Let $\sigma\in H$ have order $\ell$, and suppose $\cl_G(\sigma)\not\subset H$. The $\ell$-divisibility hypothesis guarantees some $G_1\leq G$ containing $H$, such that $\cl_{G_1}(\tau)\cap H=\emptyset$ for some $\tau\in\cl_G(\sigma)\cap G_1$. Observe that for any $n$ we have $\tau^n\in\cl_G(\sigma^n)\cap G_1$. Furthermore, if $\ell\nmid n$ then $\pi\tau^n\pi^{-1}\in H$ for some $\pi\in G_1$ would imply $\pi\tau\pi^{-1}\in H$, thereby violating $\cl_{G_1}(\tau)\cap H=\emptyset$. Thus $\cl_{G_1}(\tau^n)\cap H=\emptyset$, meaning $G_1$ satisfies the condition of Definition \ref{def:divisible} for every nontrivial element in $\langle\sigma\rangle$. In particular, $\delta_\ell(G/H)$ is at most the number of distinct subgroups of $H$ of order $\ell$, plus one more (the whole group $G$) to account for any $\sigma\in G$ of order $\ell$ with a conjugacy class disjoint from $H$.\end{proof}

\begin{corollary}Along with notation from Corollary~\ref{cor:inftower}, let $d'$ be the degree of the Galois closure of $K/F$ over $\bQ$. If $K/F$ is $\ell$-divisible for every prime number $\ell$ and $\omega_F(\fD)\geq 4d'\omega_\bQ(d)$, the $\ell$-class field tower of $K$ is infinite for some prime $\ell\,|\,d$.\end{corollary}

\begin{proof}First observe that $\rk_\ell\,\cO_K^*$ is $\rk\,\cO_K^*$ or $1+\rk\,\cO_K^*$ depending on whether $K$ contains an $\ell^\text{th}$ root of unity. Either way, $\rk_\ell\,\cO_k^*\leq[K:\bQ]$. We also have $e_\ell(d)\leq\log_\ell(d)\leq\log_2([K:\bQ])$. Combining these upper bounds, we see that \begin{equation}\label{eq:4d}2+2\sqrt{\rk\,\cO_K^*+1}+\rk_\ell\,\cO_K^*-\rk_\ell\,\cO_F^*+e_\ell(d)\leq 4[K:\bQ]\end{equation} whenever $[K:\bQ]>1$. Thus the desired inequality follows from Corollary~\ref{cor:inftower} provided $\max_\ell\delta_\ell(K/F)\leq [L:K]$, where $L$ is the Galois closure of $K/F$. An even stronger bound on $\max_\ell\delta_\ell(K/F)$ is given in the preceding lemma.\end{proof}

Finally, we turn to our density result on nilpotent extensions with infinite class field towers (Theorem~\ref{thm:asymptotic} and Corollary~\ref{cor:density}). We require two lemmas, both of which are already ``known" in the sense that they are straightforward consequences of well-known results that are common to textbooks. (For Lemma \ref{lem:numfieldbound} see \cite{neukirch}, especially sections II.5 and VI.6, and \cite{milne}, especially section V.1. For Lemma~\ref{lem:asymptotic} see \cite{tenenbaum}, especially Section II.6.1 for an exposition of arguments due to Landau \cite{landau}.) Proofs are included, however, as the author has been unable to find either assertion in the literature.

\begin{lemma}\label{lem:numfieldbound}Fix a prime $\ell\in\bZ$. Let $S_0$ and $S_\infty$ be finite sets of finite and infinite primes, respectively, of a number field $F$, and let $T_0$ be any subset of the set of primes in $S_0$ that lie above $\ell$ (if there are any). For each $\fp\in S_0$, let $$\varepsilon_\fp =\begin{cases}1 & |\cO_F/\fp|\equiv 1\,\textup{mod}\,\ell\\ 1+[F_\fp:\bQ_\ell] & \fp\in T_0\\ 0 & \text{otherwise.}\end{cases}$$ Also let $\varepsilon_\infty$ equal $|S_\infty|$ if $\ell=2$ and $0$ if $\ell\neq 2$, and set \begin{equation}\label{eq:defexp}e=\rk_\ell\,\Cl_F + \varepsilon_\infty + \sum_{\fp\in S_0}\varepsilon_\fp.\end{equation} Then there are at most $(\ell^e-1)/(\ell-1)$ cyclic extensions of $F$ of degree $\ell$ that are unramified outside of $S_0\cup S_\infty$ and not wildly ramified outside of $T_0$.\end{lemma}

\begin{proof}Assume for the sake of contradiction that the lemma fails. Let $n=1+(\ell^e-1)/(\ell-1)$, and fix any cyclic extensions $K_1,...,K_n$ of $F$ of degree $\ell$ that are unramified outside of $S_\infty\cup S_0$ and not wildly ramified outside of $T_0$. Let $\fm_i$ denote the conductor of $K_i/F$, which is supported on $S_0\cup S_\infty$ by assumption. Each finite prime $\fp\in S_0-T_0$ is either unramified or tamely ramified in $K_i$, so for such primes we have $\fp^2\nmid \fm_i$ for each $i$ \cite[page 23]{milne} \cite[Proposition VI.6.5]{neukirch}. Thus the same holds for the least common multiple of the $\fm_i$, call it $\fm=\fm_\infty\fm_0$, where $\fm_\infty$ is supported on $S_\infty$ and $\fm_0$ is supported on $S_0$. Letting $F(\fm_i)$ and $F(\fm)$ denote the corresponding ray class fields, we have $K_i\subseteq F(\fm_i)\subseteq F(\fm)$ \cite[page 158]{milne}. Let $\Cl_\fm(F)$ denote the ray class group for the modulus $\fm$, so that $\Gal(F(\fm)/F)\simeq \Cl_\fm(F)$. Since each $K_i$ corresponds to a unique quotient of $\Gal(F(\fm)/F)$ of index $\ell$, we may produce a contradiction by proving that $\Cl_\fm(F)$ has at most $(\ell^e-1)/(\ell-1)$ such quotients.

The number of quotients of $\Cl_\fm(F)$ of index $\ell$ equals the number of quotients of $\Cl_\fm(F)/(\Cl_\fm(F))^\ell$ of index $\ell$, and this is precisely $(\ell^r-1)/(\ell-1)$, where $r=\rk_\ell\,\Cl_\fm(F)$. We claim that $\rk_\ell\,\Cl_\fm(F) \leq e$, which will complete the proof.

Let $v_\fp$ denote the valuation associated to a finite prime $\fp$. Let $F_\fm$ be the group consisting of $\alpha\in F^*$ such that $v_\fp(\alpha)=0$ for all $\fp\,|\,\fm_0$, and let $F_{\fm,1}$ consist of $\alpha\in F_\fm$ such that $v_\fp(\alpha-1)\geq v_\fp(\fm_0)$ for all $\fp$, and $\sigma(\alpha) > 0$ for all real $\sigma\,|\,\fm_\infty$. Recall (the right half of) the canonical exact sequence $F_\fm/F_{\fm,1}\to \Cl_\fm(F)\to\Cl(F)\to 1$ \cite[Theorem V.1.7]{milne}. The $\ell$-rank of an abelian group is the dimension of its tensor product with $\bZ/\ell\bZ$, so right exactness of the tensor product applied to this sequence shows that $$\rk_\ell\,\Cl_\fm(F)\leq \rk_\ell\,\Cl(F) + \rk_\ell(F_\fm/F_{\fm,1}).$$ Comparing the inequality above to (\ref{eq:defexp}), we see that our claim has been reduced to proving that \begin{equation}\label{eq:Fm}\rk_\ell(F_\fm/F_{\fm,1})\leq\varepsilon_\infty +\sum_\fp\varepsilon_\fp.\end{equation}

From the isomorphism \cite[Theorem V.1.7]{milne} $$F_\fm/ F_{\fm,1}\simeq(\cO_F/\fm_0)^*\times\{\pm1\}^{\omega(\fm_\infty)}$$ and the fact that $\omega(\fm_\infty)\leq |S_\infty|$, we see that $\varepsilon_\infty$ bounds the $\ell$-rank of $F_\fm/ F_{\fm,1}$ contributed by the factor $\{\pm 1\}^{\omega(\fm_\infty)}$. Furthermore, by writing $(\cO_F/\fm_0)^*$ as a product of the unit groups $(\cO_F/\fp^{v_\fp(\fm)})^*$, we may conclude (\ref{eq:Fm}) provided  $\rk_\ell(\cO_F/\fp^{v_\fp(\fm)})^*\leq \varepsilon_\fp$ for each $\fp\in S_0$. This is immediate if $\fp\not\in T_0$, because in this case $v_\fp(\fm)=1$, so $(\cO_F/\fp^{v_\fp(\fm)})^*$ is a cyclic group that has at most one subgroup of index $\ell$ depending on whether $\ell$ divides $|(\cO_F/\fp^{v_\fp(\fm)})^*|=|\cO_F/\fp|-1$. For $\fp\in T_0$, let $F_\fp$ be the localization of $F$ at $\fp$, and observe that the inclusion $F\hookrightarrow F_\fp$ induces an isomorphism $(\cO_F/\fp^{v_\fp(\fm)})^*\simeq \cO_{F_\fp}^*/U^{(v_\fp(\fm))}_{F_\fp}$, where $U^{(i)}_{F_\fp}=\{\alpha\in F_\fp\mid v_\fp(\alpha-1)\geq i\}$ is the usual filtration of $\cO_{F_\fp}^*$. Recalling that $\cO_{F_\fp}^*$ is the product of a cyclic group (of roots of unity) and a free group of rank $[F_{\fp}:\bQ_\ell]$ \cite[Proposition II.5.3]{neukirch}, we see that $\rk_\ell(\cO_F/\fp^{v_\fp(\fm)})^*\leq \rk_\ell\,\cO_{F_\fp}^*\leq 1 + [F_\fp:\bQ_\ell]=\varepsilon_\fp$ as claimed.\end{proof}

\begin{lemma}\label{lem:asymptotic}Fix some positive integer $j$, and for $k\in\bZ$ and $x\in\bR$ define $$S_k(x)=\{s^jt\mid s,t\in \bZ,\;\omega(s)=\omega(st)\leq k,\;|s^jt|\leq x\}.$$ Then $$\lim_{x\to\infty}\frac{|S_k(x)|\log x}{x^{1/j}(\log\log x)^{k-1}}=\frac{2j}{(k-1)!}.$$\end{lemma}

\begin{proof} Let $\pi_k(x)$ count the number of square-free positive integers with exactly $k$ distinct prime divisors. In \cite{landau}, Landau proves that $$\lim_{x\to\infty}\frac{\pi_k(x)\log x}{x(\log\log x)^{k-1}}=\frac{1}{(k-1)!}.$$

Now let $$T_k(x)=\big\{\pm s^j\mid s\in\bZ,\;\omega(s)\leq k,\;|s^j|\leq x\big\}.$$ Observe that $T_k(x)$ counts the same integers as $\pi_1(x)+\cdots +\pi_k(x)$ (which is dominated by $\pi_k(x)$ for large $x$), but raised to the power $j$ and scaled by $\pm 1$ (in case $j$ is even). It follows that $$\lim_{x\to\infty}\frac{|T_k(x)|\log x}{x^{1/j}(\log\log x)^{k-1}}=\frac{2j}{(k-1)!}.$$ So our proof will be complete if we can show that $\lim_x|S_k(x)|/|T_k(x)|=1$. This holds trivially for $k=0$. For $k>0$ we induct on $k$ to obtain \begin{align*}|S_k(x)-T_k(x)|&\leq\sum_{\substack{i > j\\p<x^{1/i}}}\!\bigg|S_{k-1}\!\bigg(\frac{x}{p^i}\bigg)\bigg|\\&\leq c_0(j,k)\!\!\sum_{\substack{i > j\\p<x^{1/i}}}\!\frac{x^{1/j}(\log\log x)^{k-2}}{p^{i/j}\log x}\hspace{\parindent}\text{for some }c_0(j,k)\text{ by induction}\\ &\leq\frac{c_0(j,k)x^{1/j}(\log\log x)^{k-2}}{\log x}\sum_{\substack{i > j\\p\geq 2}}\frac{1}{p^{i/j}}\\&\leq\frac{c_1(j,k)x^{1/j}(\log\log x)^{k-2}}{\log x}\hspace{\parindent}\text{for some }c_1(j,k).\end{align*} The last expression is negligible compared to $|T_k(x)|$ as $x$ grows.\end{proof}

For our final theorem, recall Notation~\ref{not:malle}.

\begin{theorem}\label{thm:malle}Let $F$ be a number field, and let $G/H$ be groups with $G$ nilpotent. Set $d=[F:\bQ][G:H]$ and \begin{equation}\label{eq:defe}e=\omega_\bQ([G:H])\!\max_{\ell\,|\,[G:H]}\!\Big\lfloor\delta_\ell(G/H)\big(d+2\sqrt{d}+e_\ell([G:H])+2-\rk_\ell\,\cO_F^*\big)\Big\rfloor.\end{equation} There is an constant $c(d)$ (depending only on $d$) such that for large $x$, $$N_0(F,G/H,x)< \frac{c(d)x^{a(G/H)}(\log\log x)^{e-1}}{\log x}.$$ In particular, if $H=1$ is the trivial subgroup, then there is a constant $c'(F,G)$ (depending only on $F$ and $G$) such that for large $x$, $$\frac{N_0(F,G/1,x)}{N(F,G/1,x)}<\frac{c'(F,G)(\log \log x)^{e-1}}{\log x}.$$\end{theorem} 

\begin{proof}Recall that $G$ is a direct product of its Sylow subgroups, and $H$ is a direct product of subgroups of the Sylow subgroups of $G$. In particular, we can fix a normal series $H=H_0\trianglelefteq H_1\trianglelefteq\cdots\trianglelefteq H_n=G$ such that $[H_i:H_{i-1}]$ is prime for each $i$ (as this can evidently be done in each individual Sylow subgroup factor of $G$). Also fix some $D\in \bZ$, and let $\Sigma=\Sigma_D$ be the set of extensions of $F$ (in some fixed algebraic closure) with Galois group $G/H$, relative discriminant of absolute norm $D$, and finite class field tower. Note that if such a field exists, call one $K/F$ with relative discriminant $\fD$, then Corollary~\ref{cor:inftower} bounds $\omega_F(\fD)$ by $e$ from (\ref{eq:defe}) (using $\rk_\ell\,\cO_K^*\leq\rk_\ell\,\cO_K^*+1\leq n$). Since every prime dividing $D$ lies below at least one prime ideal dividing $\fD$, we conclude that \begin{equation}\label{eq:omegabound}\Sigma_D\neq\emptyset\,\text{ implies }\,\omega_\bQ(D)\leq\omega_F(\fD)\leq e.\end{equation}

We define subsets of $\Sigma$ as follows: for each $i\leq n$, let $\Sigma_i$ be the set of fixed fields of $H_i$ in some $K\in\Sigma$. So an elements of $\Sigma_{i-1}$ are extensions of elements of $\Sigma_i$ of prime degree $[H_i:H_{i-1}]$. As a subfield of a field with finite class field tower, each $K_i\in\Sigma_i$ must also have a finite class field tower. So the Golod--Shafarevich theorem bounds $\rk_\ell\,\Cl(K_i)$ by a function of $d$. Thus by Lemma~\ref{lem:numfieldbound}, the number of fields in $\Sigma_{i-1}$ that contain a given $K_i$ are bounded by a function of $d$ (and $\omega_F(\fD)$, but as we see from (\ref{eq:omegabound}), this can be absorbed by the function of $d$). In other words, there is a function of $d$ that bounds $|\Sigma_{i-1}|/|\Sigma_i|$. By induction, the same is true of $|\Sigma_0|/|\Sigma_n|=|\Sigma|/1$, meaning there is some $c_0(d)$ such that \begin{equation}\label{eq:sigmabound}|\Sigma_D|<c_0(d)\end{equation} for any choice of $D$.

Now let $a=a(G/H)$, and recall from Notation~\ref{not:malle} that $1/a$ is a positive integer. If $D$ is a discriminant norm for which $\Sigma_D\neq\emptyset$, the restriction that $K/F$ have Galois group $G/H$ forces $D$ to be of the form $s^{1/a}t$, where $s,t\in \bZ$ and $\omega_\bQ(s)=\omega_\bQ(D)$. (This observation is the impetus for the weak form of Malle's conjecture \cite{malle}. See Section 4 in \cite{malle2}.) In other words, if in the statement of Lemma~\ref{lem:asymptotic} we set $j=1/a$, then $\Sigma_D\neq\emptyset$ implies $D$ belongs to the set $S_e(x)$ for any $x\geq |D|$.

The combination of Lemma~\ref{lem:asymptotic} and (\ref{eq:sigmabound}) gives  $$N_0(F,G/H,x)=\sum_{|D|\leq x}\!|\Sigma_D|\leq \!\!\sum_{D\in S_e(x)}\!\!\!\!c_0(d)=c_0(d)|S_e(x)|\leq\frac{c(d)x^a(\log\log x)^{e-1}}{\log x}$$ for some choice of $c(d)$ (provided $\log\log x > 0$). This proves the first inequality of our theorem.

The second inequality is an immediate consequence of the first along with the theorem of Kl{\"u}ners and Malle that states $N(F,G/1,x)> c_1(F,G)x^a$ when $G$ is nilpotent \cite{kluners}.\end{proof}

Note that the value of $e$ in (\ref{eq:defe}) differs from that in Theorem~\ref{thm:asymptotic} and Corollary~\ref{cor:density}. In the latter two results, the bound in Corollary~\ref{cor:inftower} is no longer directly used to define $e$. To the expression on the right-hand side of Corollary~\ref{cor:inftower}'s inequality, we apply (\ref{eq:4d}) and the bound on $\max_\ell(\delta_\ell(K/F))$ from Corollary~\ref{cor:nilpotent}. This provides the exponent used in the introduction.

\section{Appendix: Divisible groups and field extensions}\label{sec:4}

This section identifies a few types of fields extensions that are $\ell$-divisible and strongly $\ell$-divisible. Our list is not comprehensive. For example, we do not prove that any extensions $K/F$ with simple Galois groups are $\ell$-divisible (other than the trivial case when $\ell$ does not divide the degree of the Galois closure of $K/F$ over $K$). Such extensions do exist, however. The alternating group $A_5$ has one nontrivial subgroup $H$ up to conjugation such that $A_5/H$ is $\ell$-divisible for all primes $\ell$, namely the subgroup of order $2$, while $A_6$ has $8$ such subgroups up to conjugation and $A_7$ has $15$.

\subsection{Towers of fields} This section, specifically Corollary~\ref{cor:galoistower}, justifies Corollary~\ref{cor:intro1} from the introduction.

\begin{lemma}\label{lem:grouptower}Let $H=H_0\leq H_1\leq\cdots\leq H_n=G$ be groups. If $H_i/H_{i-1}$ is $\ell$-divisible for all $i$, then $G/H$ is $\ell$-divisible with $$\delta_\ell(G/H)\leq\sum_{i=1}^n\delta_\ell(H_i/H_{i-1}).$$ The analogous claim holds for strong $\ell$-divisibility with $$\Delta_\ell(G/H)\leq\sum_{i=1}^n[H_{i-1}:H]\Delta_\ell(H_i/H_{i-1}).$$\end{lemma}

\begin{proof}Let us prove the case $n=2$, then the full claim follows by induction. By definition of $\delta_\ell(G/H)$ (or $\Delta_\ell(G/H)$), it suffices to show that the union of two sets of intermediate groups satisfying Definition~\ref{def:divisible} (or Definition~\ref{def:Divisible}), one set for $G/H_1$ and the other for $H_1/H$, will satisfy Definition~\ref{def:divisible} (or Definition~\ref{def:Divisible}) for $G/H$. 
    
Suppose $\sigma\in G-H$ has order $\ell$. On the one hand, if $\cl_G(\sigma)\not\subset H_1$, then $\ell$-divisibility of $G/H_1$ provides some $G'\leq G$ containing $H_1$ such that $\cl_{G'}(\tau)\cap H_1=\emptyset$ for some $\tau\in G'\cap \cl_G(\sigma)$. Then $H\leq H_1\leq G'$ and $\cl_{G'}(\tau)\cap H\subseteq\cl_{G'}(\tau)\cap H_1=\emptyset$ as desired. On the other hand, if $\cl_G(\sigma)\subset H_1$, then $\ell$-divisibility of $H_1/H$ provides some $H'\leq H_1$ containing $H$ such that $\cl_{H'}(\tau)\cap H=\emptyset$ for some $\tau\in H'\cap \cl_{H_1}(\sigma)\subset H'\cap \cl_G(\sigma)$.

 The same argument works for strong $\ell$-divisibility.\end{proof}

\begin{lemma}\label{lem:quo}Let $N\leq H\leq G$ with $N\trianglelefteq G$, and let $\cG=G/N$ and $\cH=H/N$. Then $G/H$ is $\ell$-divisible if and only if $\cG/\cH$ is $\ell$-divisible, in which case $\delta_\ell(G/H)=\delta_\ell(\cG/\cH)$.  The analogous claim holds for strong $\ell$-divisibility with $\Delta_\ell(G/H)=\Delta_\ell(\cG/\cH)$.\end{lemma}

\begin{proof}The lattice isomorphism theorem gives a bijective correspondence between subgroups $H\leq G_1,...,G_n\leq G$ and subgroups $\cH\leq\cG_1,...,\cG_n\leq \cG$, where $\cG_i=G_i/N$. Since $N\leq H$, for any $\sigma\in G$ we have $\cl_G(\sigma)\subset H$ if and only if $\cl_{\cG}(\sigma N)\subset \cH$, and $\cl_{G_i}(\sigma)\cap H=\emptyset$ if and only if $\cl_{\cG_i}(\sigma N)\cap \cH=\emptyset$. These two equivalences show that Definition~\ref{def:divisible} or Definition~\ref{def:Divisible} holds for $H\leq G_1,...,G_\delta\leq G$ if and only if it holds for $\cH\leq\cG_1,...,\cG_\delta\leq\cG$. Thus $\delta_\ell(G/H)=\delta_\ell(\cG/\cH)$ and, since $[G_i:H] = [\cG_i:\cH]$, $\Delta_\ell(G/H)=\Delta_\ell(\cG/\cH)$.\end{proof}

The last lemma shows that $G/H$ in Definitions~\ref{def:divisible} and \ref{def:Divisible} could be replaced with $\Gal(L/F)$, where $L$ is \textit{any} normal extension of $F$ containing $K$, not necessarily the normal closure of $K/F$. This allows us to rephrase Lemma~\ref{lem:grouptower} in terms of fields. 

\begin{theorem}\label{thm:fieldtower}Let $F=F_0\subseteq F_1\subseteq\cdots\subseteq F_n=K$ be number fields. If $F_i/F_{i-1}$ is $\ell$-divisible for all $i$, then $K/F$ is $\ell$-divisible with $$\delta_\ell(K/F)\leq\sum_{i=1}^n\delta_\ell(F_i/F_{i-1}).$$ The analogous claim holds for strong $\ell$-divisibility with $$\Delta_\ell(K/F)\leq\sum_{i=1}^n[K:F_i]\Delta_\ell(F_i/F_{i-1}).$$\end{theorem}

\begin{proof}Let $L$ denote the Galois closure of $K/F$, and let $L_i$ denote the Galois closure of $F_i/F_{i-1}$. Also let $G/H=\Gal(K/F)$, and let $H_i/1=\Gal(L/F_i)$. Finally, let $N_i/1=\Gal(L/L_i)$ so that $N_i\trianglelefteq H_i$. The Galois group of $L_i/F_{i-1}$ is isomorphic to $H_{i-1}/N_i$, in which the Galois group of $L_i/F_i$ is embedded as $H_i/N_i$. Our hypothesis is that $F_i/F_{i-1}$, and thus by definition $(H_{i-1}/N_i)/(H_i/N_i)$, is (strongly) $\ell$-divisible. So by Lemma~\ref{lem:quo}, $H_{i-1}/H_{i}$ is (strongly) $\ell$-divisible with $\delta_\ell(H_{i-1}/H_{i}) = \delta_\ell(F_i/F_{i-1})$ and/or $\Delta_\ell(H_{i-1}/H_{i}) = \Delta_\ell(F_i/F_{i-1})$. Thus Lemma~\ref{lem:grouptower} applies, and we conclude that $G/H$ is (strongly) $\ell$-divisible with $$\delta_\ell(K/F)= \delta_\ell(G/H)\leq\sum_{i=1}^n\delta_\ell(H_{i-1}/H_i)=\sum_{i=1}^n\delta_\ell(F_i/F_{i-1})$$ and/or $$\Delta_\ell(K/F)= \Delta_\ell(G/H)\leq\sum_{i=1}^n[H_i:H]\Delta_\ell(H_{i-1}/H_i)=\sum_{i=1}^n[K:F_i]\Delta_\ell(F_i/F_{i-1}).$$ (Note that $H_i\leq H_{i-1}$ unlike the ordering in Lemma~\ref{lem:grouptower}.)\end{proof}

\begin{proposition}\label{prop:Galois}Suppose $K/F$ is a Galois extension of number fields, and let $\sigma_1,...,\sigma_n$ be generators for the distinct subgroups of order $\ell$ in $\Gal(K/F)$. If $n\geq 1$, then $\delta_\ell(K/F) = 1$ and $\Delta_\ell(K/F)\leq\min(n,\frac{1}{\ell}|\langle\sigma_1,...,\sigma_n\rangle|)$.\end{proposition}

\begin{proof}In this scenario, $H$ in Definitions~\ref{def:divisible} and \ref{def:Divisible} is the trivial group. Thus the subgroup $\langle\sigma_1,...,\sigma_n\rangle$ shows that $\delta_\ell(K/F)= 1$ and $\Delta_\ell(K/F)\leq \frac{1}{\ell}|\langle\sigma_1,...,\sigma_n\rangle|$. The subgroups $\langle\sigma_1\rangle,...,\langle\sigma_n\rangle$ show that $\Delta_\ell(K/F)\leq n$.\end{proof} 

Combining the result above with Theorem~\ref{thm:fieldtower} immediately proves the following.

\begin{corollary}\label{cor:galoistower}Let $F=F_0\subseteq F_1\subseteq\cdots\subseteq F_n=K$ be number fields. If $F_i/F_{i-1}$ is Galois for all $i$, then $K/F$ is strongly $\ell$-divisible with $\delta_\ell(K/F)\leq n$.\qed\end{corollary}

\subsection{Nilpotent extensions} As a consequence of Lemma~\ref{lem:grouptower}, if $G$ has a subnormal series containing $H$, then $G/H$ is strongly $\ell$-divisible. In this section we ask when a subnormal series that may not contain $H$ guarantees $\ell$-divisibility or strong $\ell$-divisibility. The dihedral groups $D_n$ with $4\nmid n$ show that $G$ being solvable is not sufficient. But as Corollary \ref{cor:nilpotent} proves, nilpotence is sufficient.

\begin{lemma}\label{lem:syls}Let $H\leq G$ be groups. If $G/H$ is $\ell$-divisible and $\ell\nmid [G\!:\!H]$, then $H$ contains all Sylow $\ell$-subgroups of $G$.\end{lemma}

\begin{proof}We will use induction on $n$ to prove that if $\sigma\in G$ has order $\ell^n$ for some $n$, then $\sigma\in H$. Consider the case $n=1$. (We begin here rather than $j=0$ because this case is used to prove the induction step.) Suppose $\sigma\in G$ has order $\ell$, and assume for the sake of contradiction that $\sigma\not\in H$. By $\ell$-divisibility, there exists $G'\leq G$ containing $H$ such that $\cl_{G'}(\tau)\cap H=\emptyset$ for some $\tau\in \cl_G(\sigma)\cap G'$. Since $\cl_{G'}(\tau)$ intersects every Sylow $\ell$-subgroup of $G'$ nontrivially while being disjoint from $H$, a Sylow $\ell$-subgroup of $H$ cannot also be a Sylow $\ell$-subgroup of $G_1$. In other words, $\ell\;|\;[G'\!:\!H]$, which contradicts $\ell\nmid [G\!:\!H]$. 
    
For the induction step, fix some $n\geq 1$, and assume $\sigma\in H$ for all $\sigma\in G$ with $|\sigma|=\ell^n$. In particular, the subgroup generated by a such $\sigma$ is contained in $H$. Call this subgroup $N$ and observe that $N\trianglelefteq G$. Let $\cG=G/N$ and $\cH=H/N$. Consider some $\sigma\in G$ with $|\sigma|=\ell^{n+1}$. If $\sigma\in N\leq H$ we are done. Otherwise, since $\sigma^{\ell}\in N$, we see that $\sigma N$ has order $\ell$ in $\cG$. Note that $[\cG\!:\!\cH]=[G\!:\!H]$, which is not divisible by $\ell$. Furthermore, Lemma~\ref{lem:quo} tells us $\cG/\cH$ is $\ell$-divisible. Thus the case $n=1$ from the previous paragraph applies to $\sigma N\in \cG$, the conclusion being $\sigma N\in \cH$. This implies $\sigma\in NH = H$.\end{proof}

\begin{lemma}\label{lem:oneway}Let $H,N\leq G$ with $N$ normal, and let $\cG=G/N$ and $\cH=HN/N$. If $\cG/\cH$ and $HN/H$ are $\ell$-divisible, then so is $G/H$ with $\delta_\ell(G/H)\leq\delta_\ell(\cG/\cH)+\delta_\ell(HN/H)$. The analogous claim holds for strong divisibility with $\Delta_\ell(G/H)\leq \Delta_\ell(HN/H)+[HN:H]\Delta_\ell(\cG/\cH)$.\end{lemma}

\begin{proof}By Lemma~\ref{lem:quo}, $\delta_\ell(G/HN)=\delta_\ell(\cG/\cH)$ and $\Delta_\ell(G/HN)=\Delta_\ell(\cG/\cH)$. The claim then follows by applying Lemma~\ref{lem:grouptower} to $H\leq HN\leq G$.\end{proof}

Note that if $N$ normalizes $H$ then strong $\ell$-divisibility of $HN/H$ is automatic. As a special case of this, we see that (strong) $\ell$-divisibility is preserved by central extensions, meaning if $N$ is in the center of $G$, (strong) $\ell$-divisibility of $G/H$ is implied by that of $\cG/\cH$. Hence the next theorem, which justifies Corollary~\ref{cor:intro2} from the introduction.

\begin{theorem}\label{thm:normalS}Let $G$ be a group and $\ell$ a prime. The following are equivalent: \begin{enumerate}[label=(\alph*)]
\item $G/H$ is strongly $\ell$-divisible for all subgroups $H\leq G$,
\item $G/H$ is $\ell$-divisible for all subgroups $H\leq G$,
\item $G/S$ is strongly $\ell$-divisible for some Sylow $\ell$-subgroup $S\leq G$,
\item $G/S$ is $\ell$-divisible for some Sylow $\ell$-subgroup $S\leq G$,
\item $G$ has a unique Sylow $\ell$-subgroup.\end{enumerate}
If these conditions hold and $1=S_0\trianglelefteq S_1\trianglelefteq\cdots \trianglelefteq S_n =S$ is a central series for the Sylow $\ell$-subgroup $S\leq G$, then $\delta_\ell(G/H)$ is at most the number of proper containments in the chain $HS_0\leq\cdots\leq HS_n$.\end{theorem}

\begin{proof} Among (a)--(d), all are evidently implied by (a), and all evidently imply (d). Furthermore, that (d) implies (e) is a direct application of Lemma~\ref{lem:syls}. So to prove that (a)--(e) are equivalent, it suffices to show that (e) implies (a).

Assume (e) holds. Let $S$ be the unique Sylow $\ell$-subgroup of $G$, and let $1 =S_0 \trianglelefteq S_1 \trianglelefteq \cdots \trianglelefteq S_n = S$ be a central series for $S$. Letting $\cG_i=G/S_i$ and $\cH_i=HS_i/S_i$ (note that each $S_i$ is a characteristic subgroup of $G$ and is therefore normal), we claim that $\cG_i/\cH_i$ is strongly $\ell$-divisible for all $i$, and that \begin{equation}\label{eq:delta1}\delta_\ell(\cG_i/\cH_i)\leq\begin{cases}\delta_\ell(\cG_{i+1}/\cH_{i+1}) & HS_i=HS_{i+1}\\ \delta_\ell(\cG_{i+1}/\cH_{i+1})+1 & HS_i\neq HS_{i+1}.\end{cases}\end{equation} Once proved, these claims imply the entire theorem.

We use induction in reverse order on $i$. The base case is $i=n$. We have $\cG_n=G/S$, which has no elements of order $\ell$. This makes strong $\ell$-divisibility automatic with $\delta_\ell(\cG_n/\cH_n) = 0$. For the induction step, assume $\cG_{i+1}/\cH_{i+1}$ is strongly $\ell$-divisible. Set $\cN=S_{i+1}/S_i$ so that the canonical isomorphism $\cG_{i+1}\to \cG_i/\cN$ maps $\cH_{i+1}$ to $\cH_i\cN/\cN$. By definition of a central series, $\cN$ lies in the center of $\cG_i$. In particular, $\cH_i$ is normal in $\cH_i\cN$, so either $\delta(\cH_i\cN/\cH_i)=0$ if $\ell\nmid[\cH_i\cN:\cH_i]$ or $\delta(\cH_i\cN/\cH_i)=1$ if $\ell\,|\,[\cH_i\cN:\cH_i]$ (the latter being the group-theoretic phrasing of Proposition~\ref{prop:Galois}). Therefore, by Lemma~\ref{lem:oneway} and the induction hypothesis, $\cG_i/\cH_i$ is strongly $\ell$-divisible with \begin{equation}\label{eq:delta2}\delta_\ell(\cG_i/\cH_i)\leq\begin{cases}\delta(\cG_{i+1}/\cH_{i+1}) & \ell\nmid[\cH_i\cN\!:\!\cH_i]\\ \delta(\cG_{i+1}/\cH_{i+1})+1 & \ell\;|\;[\cH_i\cN\!:\!\cH_i].\end{cases}.\end{equation} Since $[\cH_i\cN\!:\!\cH_i] = [HS_{i+1}\!:\!HS_i]$ (which divides $[S_{i+1}\!:\!S_i]$ and is therefore a power of $\ell$), we see that (\ref{eq:delta1}) and (\ref{eq:delta2}) agree.\end{proof}

\begin{corollary}\label{cor:nilpotent}Let $G/H=\Gal(K/F)$. If $G$ is nilpotent, then $K/F$ is strongly $\ell$-divisible for every prime $\ell$ with $\delta_{\ell}(K/F)$ at most the nilpotency class of $G$.\end{corollary}

\begin{proof}This follows immediately from Theorem~\ref{thm:normalS} and the fact that a group $G$ is nilpotent if and only if each of its Sylow subgroups is unique, in which case the nilpotency class of each Sylow subgroup is at most that of the group.\end{proof}

\subsection{Semidirect products} Let $D_n=\langle r,s\mid r^n,s^2,(rs)^2\rangle$, the dihedral group of order $2n$. The introduction claims that $D_n/\langle s\rangle$ (or, equivalently, $D_n/\langle r^is\rangle$ for any $i$) is 2-divisible if and only if $4\,|\,n$. The theorem below identifies what separates these two cases: the automorphism of $\langle r\rangle$ defined by $r^i\mapsto r^{-i}$ restricts to the identity map on the Sylow 2-subgroup of $\langle r\rangle$ precisely when $4\nmid n$.

For a nilpotent group $G$, we use $\nil(G)$ to denote the nilpotency class of $G$.

\begin{theorem}\label{thm:semidirect}Let $G=N\rtimes_{\varphi}H$ for some $\varphi:H\to\textup{Aut}(N)$. If $\varphi(H)$ and $N$ have unique Sylow $\ell$-subgroups $S$ and $T$ such that the restriction map $\varphi(h)\mapsto\varphi(h)|_T$ is injective on $S$, then $G/H$ is $\ell$-divisible with $\delta_\ell(G/H)\leq \nil(S)\nil(T)+1$.\end{theorem} 

\begin{proof}Throughout the proof, we do not distinguish $H$ and $N$ from their canonical copies in $G$. 

Let $n_S=\nil(S)$, $n_T=\nil(T)$, and $n=n_Sn_T$. We will prove by induction on both $n_S$ and $n_T$ that there exist subgroups $G_1,...,G_{n_Sn_T}\leq G$ containing $H$ such that for any $\sigma\in H$ of order $\ell$ satisfying $\cl_G(\sigma)\not\subset H$, there exists some $G_i$ and $\tau\in\cl_G(\sigma)\cap G_i$ with $\cl_{G_i}\cap H=\emptyset$. Once proved, this implies the theorem: the only remaining $\sigma\in G-H$ of order $\ell$ are those (that may not exist) for which $\cl_G(\sigma)\cap H=\emptyset$, so at most one additional subgroup is needed to satisfy Definition~\ref{def:divisible}, namely $G$. 

We begin by observing that no generality is lost in assuming $N=T$. Indeed, if $\sigma\in H$ has order $\ell$, then $\varphi(\sigma)\in S$. So if $\cl_G(\sigma)\not\subset H$, which is equivalent to $\varphi(\sigma)\neq\text{Id}_N$, then by hypothesis $\varphi(\sigma)|_T\neq\text{Id}_T$, which is equivalent to $\cl_T(\sigma)\not\subset H$. But then if we find a subgroup $G_i$ of $T\rtimes_{\varphi}H$ for which $\cl_{G_i}(\tau)\cap H=\emptyset$ for some $\tau\in G_i$ that is $T\rtimes_{\varphi}H$-conjugate to $\sigma$, we have achieved the same thing in the potentially larger group $N\rtimes_{\varphi}H$ by taking the same $G_i$ and $\tau$. As such, we henceforth assume $G=T\rtimes_{\varphi}H$. This has the affect of making the injectivity hypothesis in the theorem statement vacuous for $G$ and its subgroups and quotients that are to be used in the induction proof.

The base case is $n_S=0$ or $n_T=0$, meaning $S$ or $T$ is the trivial group. Then there are no elements $\sigma\in H$ of order $\ell$ for which $\cl_G(\sigma)\not\subset H$, so we need $0=n_Sn_T$ subgroups, as claimed.

Now suppose $n_S,n_T\geq 1$, and consider some $\sigma\in H$ of order $\ell$ such that $\varphi(\sigma)$ does not descend to the identity on $T/Z(T)$, where $Z(T)$ is the center of $T$. This is equivalent to $\cl_{\cG}(\sigma)\not\subset H$, where $\cG=T/Z(T)\rtimes_{\varphi}H$. By the induction hypothesis, at least one from among $\nil(S)\nil(T/Z(T))=n_S(n_T-1)$ select subgroups of $\cG$ containing $H$, call it $\cG_i$, must satisfy $\cl_{\cG_i}(\tau Z(T))\cap H=\emptyset$ for some $\tau Z(T)\in\cl_\cG(\sigma Z(T))\cap \cG_i$. Assume without loss of generality that $\tau\in \cl_G(\sigma)$ (if not, pick a different coset representative of $\tau Z(T)$), and pull back $\cG_i$ via the quotient map $G\to \cG$. This gives the desired $G_i\leq G$ containing $H$ and $\tau$ with $\cl_{G_i}(\tau)\cap H=\emptyset$.

The previous paragraph produced $n_S(n_T-1)$ subgroups of $G$ out of our allotment of $n_Sn_T$. So we must now show that at most $n_S$ subgroups of $G$ are needed for those $\cl_G(\sigma)$ for which $\varphi(\sigma)$ descends to the identity on $T/Z(T)$. Let $R=\bZ/|T|\bZ$, let $R[\varphi(H)]$ denote the group ring written with additive notation, and let $I$ denote the kernel of the projection map $R[\varphi(H)]\to R[\varphi(H)/Z(S)]$. View $Z(T)$ as a left $R[\varphi(H)]$-module via $$\sum_{\sigma\in H}a_\sigma\varphi(\sigma)\cdot \tau = \prod_{\sigma\in H}\varphi(\sigma)(\tau^{a_\sigma}).$$ Finally, let $\Ann(I)$ be the $R[\varphi(H)]$-submodule of $Z(T)$ consisting of elements annihilated by $I$. 

We consider two cases. First suppose $\sigma\in H$ (of order $\ell$ with $\cl_G(\sigma)\not\subset H$) is such that $\varphi(\sigma)$ does not restrict to the identity on $\Ann(I)$. By definition of $I$, $\Ann(I)$ is a left $R[\varphi(H)/Z(S)]$-module. Thus we have a group $\cG=\Ann(I)\rtimes_{\varphi}\cH$, where $\cH=H/\varphi^{-1}(Z(S))$. The assumption that $\varphi(\sigma)$ is not the identity on $\Ann(I)$ means $\cl_{\cG}(\sigma\varphi^{-1}(Z(S)))\not\subset\cH$. So as before, we pull back the $\nil(S/Z(S))\nil(\Ann(I))=n_S-1$ subgroups of $\cG$ containing $\cH$ provided by the induction hypothesis. At least one of these, say $G_i\leq G$, is such that $\cl_{G_i}(\tau)\cap H=\emptyset$ for some $\tau\in\cl_G(\sigma)\cap G_i$. 

We have thus far produced $n_Sn_T-1$ subgroups of $G$ out of our allotment of $n_Sn_T$. So we must show that at most one subgroup of $G$ is needed for the second case: when $\sigma$ is such that $\varphi(\sigma)$ restricts to the identity on $\Ann(I)$ (and still $\varphi(\sigma)$ descends to the identity on $T/Z(T)$). We claim that $G'=\Ann(I)\rtimes_{\varphi}H$ is the only subgroup needed. To see this, let $\tau\in\cl_G(\sigma)\cap G'$ be any element of the form $\pi\sigma\pi^{-1}$ for some $\pi\in T$ satisfying $1\neq \pi\,\varphi(\sigma)(\pi^{-1})\in\Ann(I)$. Assuming such a $\pi$ exists, let us check that $\cl_{G'}(\tau)\cap H=\emptyset$. Elements of $G'$ are of the form $\pi'\sigma'$ for $\pi'\in \Ann(I)$ and $\sigma'\in H$. We have \begin{align*}(\pi'\sigma')\tau(\pi'\sigma')^{-1} &= (\pi'\sigma')\Big(\pi\,\varphi(\sigma)(\pi^{-1})\,\sigma\Big)\Big(\varphi(\sigma'^{-1})(\pi'^{-1})\,\sigma'^{-1}\Big)\\ &= \pi'\,\varphi(\sigma')(\pi\varphi(\sigma)(\pi^{-1}))\,\varphi(\sigma'\sigma\sigma'^{-1})(\pi'^{-1})\,\sigma'\sigma\sigma'^{-1}.\end{align*} This belongs to $H$ if and only if $\pi'\,\varphi(\sigma')(\pi\varphi(\sigma)(\pi^{-1}))\,\varphi(\sigma'\sigma\sigma'^{-1})(\pi'^{-1})=1$, which is equivalent to $\pi\,\varphi(\sigma)(\pi^{-1}) = \varphi(\sigma'^{-1})(\pi'^{-1}\varphi(\sigma'\sigma\sigma'^{-1})(\pi'))$. But $\pi\,\varphi(\sigma)(\pi^{-1})\neq 1$ by choice of $\pi$, while \begin{align*}\varphi(\sigma'^{-1})(\pi'^{-1}\varphi(\sigma'\sigma\sigma'^{-1})(\pi'))&= \varphi(\sigma'^{-1})(\pi'^{-1})\,\varphi(\sigma)(\varphi(\sigma'^{-1})(\pi'))\\&=(1-\varphi(\sigma))\cdot\varphi(\sigma^{-1})(\pi'^{-1})\\&\in (1-\varphi(\sigma))\cdot \Ann(I),\end{align*} which is the trivial submodule since $\varphi(\sigma)$ is the identity on $\Ann(I)$. Thus $G'$ works as claimed.

It remains only to prove that there exists $\pi\in T$ satisfying $1\neq \pi\,\varphi(\sigma)(\pi^{-1})\in\Ann(I)$. Start with any $\pi'\in T$ such that $\pi'\neq\varphi(\sigma)(\pi')$. We are assuming $\varphi(\sigma)$ descends to the identity on $T/Z(T)$, so $\pi'\,\varphi(\sigma)(\pi'^{-1})\in Z(T)$. Since $I$ is nilpotent, we make take $k\geq 0$ to be the smallest integer such that $I^{k+1}$ annihilates $\pi'\,\varphi(\sigma)(\pi'^{-1})$. Now, as a left $R[\varphi(H)]$-module, $I^k$ is generated by products of the form $$\prod_{i=1}^k(1-\varphi(\sigma_i))$$ for $\varphi(\sigma_i)\in Z(S)$. So minimality of $k$ implies that for some choice of $\sigma_1,...,\sigma_k$, we have \begin{equation}\label{eq:notone}\prod_{i=1}^k(1-\varphi(\sigma_i))\cdot \pi'\,\varphi(\sigma)(\pi'^{-1})\neq 1.\end{equation} But each $\varphi(\sigma_i)\in Z(S)$ commutes with $\varphi(\sigma)$. This allows us to write the element above in the form $\pi\,\varphi(\sigma)(\pi^{-1})$, where $$\pi=\prod_{i=1}^k(1-\varphi(\sigma_i))\cdot \pi'$$ (an abuse of notation since $\pi'$ may not be in $Z(T)$). We then have \begin{equation}\label{eq:one}I\cdot \pi\varphi(\sigma)(\pi^{-1})=I\prod_{i=1}^k(1-\varphi(\sigma_i))\cdot \pi'\varphi(\pi'^{-1})\subseteq I^{k+1}\cdot \pi'\,\varphi(\sigma)(\pi'^{-1})=1.\end{equation} Thus by (\ref{eq:notone}) and (\ref{eq:one}), $1\neq \pi\varphi(\sigma)(\pi^{-1})\in\Ann(I)$ as desired.\end{proof}

Observe that if $H$ and $N$ are abelian, the only nontrivial hypothesis in Theorem~\ref{thm:semidirect} is that for $h\in H$ of order $\ell$, $\varphi(h)$ is the identity on $N$ if it restricts to the identity on the Sylow $\ell$-subgroup of $N$, in which case $\delta_\ell(G/H)\leq 1\cdot 1+1=2$ as stated in Corollary~\ref{cor:intro3} in the introduction.

\begin{corollary}\label{cor:dihedral}Let $D_n=\langle r,s\,|\,r^n,s^2,(rs)^2\rangle$. Then $D_n/\langle s\rangle$ is 2-divisible if and only if $4\,|\,n$, in which case $\delta_2(D_n/\langle s\rangle) = 2$.\end{corollary}

\begin{proof}If $4\,|\,n$ then $s$ acts (by conjugation) nontrivially on $r^{n/4}$, so the hypothesis of Theorem~\ref{thm:semidirect} is satisfied. For the converse, if $G'\subseteq D_n$ contains both $s$ and some conjugate of $s$, say $r^{2i}s$ with $n\nmid 2i$, then $\cl_{G'}(r^{2i}s)$ contains all elements of the form $r^{2i+4j}s$. But if $4\nmid n$ we can choose $j$ so that $n\,|\,(2i+4j)$, which shows that $\cl_{G'}(r^{2i}s)\cap H\neq\emptyset$.\end{proof}

\subsection{Composite fields} It is false in general that if $K_1/F_1$ and $K_2/F_2$ are $\ell$-divisible then $K_1K_2/F_1F_2$ is $\ell$-divisible. For example, let $L/F_1$ be Galois with $\Gal(L/F_1)\simeq D_{12}=\langle r,\,s\mid r^12,\,s^2,\,(rs)^2\rangle$, the dihedral group of order 24. Let $K_1$ be the fixed field of $\langle s\rangle$, and let $F_2=K_2$ be the fixed field of $\langle r^2\rangle$. Then $K_1/F_1$ is 2-divisible by Corollary~\ref{cor:dihedral}, and $K_2/F_2$ is trivially 2-divisible. But the Galois group of the closure of $K_1K_2/F_1F_2$ is $\langle r^2,\,s\mid r^12,\,s^2,\,(rs)^2\rangle\simeq D_6$, and $K_1K_2=K_1$ is the fixed field of $\langle s\rangle$. Since $4\nmid 6$, $K_1K_2/F_1F_2$ is not 2-divisible by Corollary~\ref{cor:dihedral}.

There are two possible fixes: insisting that our field extensions share a common base field, or insisting on strong $\ell$-divisibility. These are the two results of this section.

\begin{theorem}\label{thm:composite}Let $F\subseteq K_1,...,K_n$ be number fields. If $K_i/F$ is $\ell$-divisible for all $i$, then $K_1\cdots K_n/F$ is $\ell$-divisible with $$\delta_\ell(K_1\cdots K_n/F)\leq \sum_{i=1}^n\delta_\ell(K_i/F).$$\end{theorem}

\begin{proof}Let $G_i/H_i=\Gal(K_i/F)$, let $G/H=\Gal(K_1\cdots K_n/F)$, and let $\varphi_i:G\to G_i$ denote the restriction to $K_i$. For each $i$, let $\delta_i=\delta(K_i/F)$, and fix subgroups $G_{i,1},...,G_{i,\delta_i}$ that satisfy Definition \ref{def:divisible}. We claim that the subgroups $\varphi^{-1}_i(G_{i,j})$ satisfy Definition \ref{def:divisible} for $G/H$.

Let $\sigma\in G-H$ have order $\ell$. Since $\sigma$ does not fix $K_1\cdots K_n$, it must not fix $K_i$ for some $i$. Thus $\varphi_i(\sigma)$ is an element of order $\ell$ in $G_i-H_i$. By assumption, there is some $j\leq \delta_i$ and $\tau'\in \cl_{G_i}(\pi_i(\sigma))\cap G_{i,j}$ such that $\cl_{G_{i,j}}(\tau')\cap H_i=\emptyset$. Now, let $\tau\in G$ be any $G$-conjugate of $\sigma$ that satisfies $\varphi_i(\tau)=\tau'$ (which exists by surjectivity of $\varphi_i$). Note that $\tau\in \varphi^{-1}_i(G_{i,j})$. Furthermore, the image under $\varphi_i$ of any $\varphi^{-1}_i(G_{i,j})$-conjugate of $\tau$ is contained in $\cl_{G_{i,j}}(\tau')$, and therefore does not fix $K_i$. Thus no $\varphi^{-1}_i(G_{i,j})$-conjugate of $\tau$ can be in $H$.\end{proof}

The lemma below fails if strong $\ell$-divisibility is replace by $\ell$-divisibility. This is why the previous theorem requires a unique base field, unlike Theorem~\ref{thm:compositestrong}.

\begin{lemma}\label{lem:compositestrong}Let $H,G'\leq G$ be groups. If $G/H$ is strongly $\ell$-divisible, then so is $G'/H\cap G'$ with $\Delta_\ell(G'/H\cap G')\leq\Delta_\ell(G/H)$.\end{lemma}

\begin{proof}Let $G_1,...,G_\delta$ satisfy Definition~\ref{def:Divisible} for $G/H$. If the order of $\sigma\in G'-H\cap G'$ is a power of $\ell$ and $\sigma^\ell\in H\cap G'$, then we must have $\sigma\in G-H$ and $\sigma^\ell\in H$. Thus for some $i\leq \delta$, $\sigma\in G_i$ and $\cl_{G_i}(\sigma)\cap H=\emptyset$. This implies $\sigma\in G_i\cap G'$ and $\cl_{G_i\cap G'}(\sigma)\cap (H\cap G')=\emptyset$. In particular, the groups $G_1\cap G',...,G_\delta\cap G'$ satisfy Definition \ref{def:Divisible} for $G'/H\cap G'$. The bound on $\Delta_\ell(G'/G'\cap H)$ follows from $[G':H\cap G']\leq[G:H]$.\end{proof}

\begin{theorem}\label{thm:compositestrong}For $i=1,...,n$, suppose $K_i/F_i$ is a strongly $\ell$-divisible extension of number fields. Then $K_1\cdots K_n/F_1\cdots F_n$ is strongly $\ell$-divisible with $$\Delta_\ell(K_1\cdots K_n/F_1\cdots F_n)\leq\sum_{i=1}^n[K_1\cdots K_n:K_1\cdots K_iF_{i+1}\cdots F_n]\Delta_\ell(K_i/F_i).$$\end{theorem}

\begin{proof}For each $i$, let $G_i/H_i=\Gal(K_i/F_i)$, let $L_i$ denote the Galois closure of $K_i/F_i$, let $L=L_1\cdots L_n$, let $\widehat{F}_i=K_1\cdots K_iF_{i+1}\cdots F_n$, let $G/1=\Gal(L/\widehat{F}_0)$ and $H/1=\Gal(L/\widehat{F}_n)$, let $\varphi_i:\Gal(L/\widehat{F}_{i-1})\to G_i$ denote the restriction, and let $G'_i$ denote the image of $\varphi_i$. 

By Lemma~\ref{lem:compositestrong}, $G'_i/H_i\cap G'_i$ is strongly $\ell$-divisible. Furthermore, we may choose subgroups $G'_{i,1},...,G'_{i,\delta_i}\leq G'$ that satisfy Definition~\ref{def:Divisible} for $G'_i/H_i\cap G'_i$ with $\frac{1}{\ell}\sum_j[G'_{i,j}:H_i\cap G'_i]\leq \Delta_\ell(G_i/H_i)=\Delta_\ell(K_i/F_i)$. We claim that the subgroups $\varphi_i^{-1}(G_{i,j})$ satisfy Definition~\ref{def:Divisible} for $G/H$. Once proved, this implies the upper bound in the theorem statement by the following argument: \begin{align*}\Delta_\ell(\widehat{F}_n/\widehat{F}_0)&\leq \frac{1}{\ell}\sum_{i=1}^n\sum_{j=1}^{\delta_i}[\pi_i^{-1}(G'_{i,j}):H]\\ &=\frac{1}{\ell}\sum_{i=1}^n[\varphi_i^{-1}(H_i\cap G'_i):H]\sum_{j=1}^{\delta_i}[\varphi_i^{-1}(G'_{i,j}):\varphi_i^{-1}(H_i\cap G'_i)]\\ &\leq \sum_{i=1}^n[\varphi_i^{-1}(H_i\cap G'_i):H]\Delta_\ell(K_i/F_i)\\ &= \sum_{i=1}^n[\Gal(L/\widehat{F}_i):\Gal(L/\widehat{F}_n)]\Delta_\ell(K_i/F_i)\\ &=\sum_{i=1}^n[\widehat{F}_n:\widehat{F}_i]\Delta_\ell(K_i/F_i).\end{align*}

Suppose the order of $\sigma\in G-H$ is a power of $\ell$ and $\sigma^\ell\in H$. Since $\sigma$ does not fix $K_1\cdots K_n$, we may choose $i\leq n$ to be the largest index such that $\sigma$ does not fix $K_i$. In particular, $\sigma\in \Gal(L/\widehat{F}_{i-1})$, and $\varphi_i(\sigma)\in G'_i-H_i\cap G'_i$. The order of $\varphi_i(\sigma)$ is also a power of $\ell$ and $\varphi_i(\sigma)^\ell\in H_i\cap G'_i$, so there exists some $j\leq \delta_i$ such that $\varphi_i(\sigma)\in G'_{i,j}$ and $\cl_{G'_{i,j}}(\varphi_i(\sigma))\cap (H_i\cap G'_i)=\emptyset$. This implies $\sigma\in\varphi_i^{-1}(G'_{i,j})$. Also, no $\varphi_i^{-1}(G'_{i,j})$-conjugate of $\sigma$ belongs to $H$ since no such conjugate fixes $K_i$; otherwise its image under $\varphi_i$ would lie in $H_i\cap G'_i$, which is impossible.\end{proof}

\printbibliography

\end{document}